\documentclass{amsart}
\usepackage{graphicx}
\usepackage{amsmath}
\usepackage{amssymb}
\usepackage{amsthm}
\usepackage{esint}

\newtheorem{theorem}{Theorem}[section]

\newtheorem{prop}[theorem]{Proposition}

\numberwithin{equation}{section}

\begin{document}

\title[Blow-up phenomena]{Blow-up phenomena for the constant scalar curvature and constant boundary mean curvature equation
(after Chen and Wu)}

\author{Pak Tung Ho}
\address{Department of Mathematics, Tamkang University, Tamsui, New Taipei City 251301, Taiwan}

\email{paktungho@yahoo.com.hk}

\author{Jinwoo Shin}
\address{Department of Mathematics \&  Research Institute of Natural Science, Sookmyung Women's University, Cheongpa-ro 47-gil 100, Youngsan-gu, Seoul, 04310, Korea}
\email{shinjin@soomkyung.ac.kr}


\date{29th of September, 2022.}


\begin{abstract}
In this paper,
the compactness
of the solutions
to the constant scalar curvature
and constant boundary mean curvature equation
is considered.
Chen and Wu constructed
a smooth counterexample
showing that the compactness
of the set of ``lower energy" solutions
to the above equation fails
when the dimension of the manifold
is not less than 62.
We prove that
a smooth counterexample still exists
when the dimension of the manifold
is not less than 35.

\end{abstract}

\maketitle

\section{Introduction}

Given an $n$-dimensional closed (i.e. compact without boundary) Riemannian manifold
$(M,g)$ with $n\geq 3$, the Yamabe problem is to find a metric $\tilde{g}$ conformal to $g$
such that the scalar curvature $R_{\tilde{g}}$ of $\tilde{g}$ is constant.
If we write $\tilde{g}=u^{\frac{4}{n-2}}g$ with $0<u\in C^\infty(M)$,
then $\tilde{g}$ has constant scalar curvature $c$ if and only if
\begin{equation}\label{0.1}
\frac{4(n-1)}{n-2}\Delta_{g} u-R_{g} u+c u^{\frac{n+2}{n-2}}=0.
\end{equation}
The Yamabe problem is solved by Aubin \cite{Aubin}, Trudinger \cite{Trudinger},
and Schoen \cite{Schoen}.
In other words, there exists at least one solution to (\ref{0.1}).

Solutions to (\ref{0.1}) are usually not unique.
In \cite{Pollack}, Pollack has used gluing techniques to prove the following:
given any conformal class with positive Yamabe constant and any positive integer $N$, there exists a new conformal
class which is close to the original one in the $C^0$-norm and contains at least $N$
metrics of constant scalar curvature.

It is an interesting question whether the set of all solutions to (\ref{0.1}) is compact in the $C^2$-topology.
A conjecture due to Schoen states
that this should be true unless $(M,g)$ is conformally equivalent to
the standard sphere (see \cite{Schoen1,Schoen2,Schoen3}).
This conjecture has been verified in low dimensions \cite{Druet,Li&Zhu,Marques} and in the locally
conformally flat case \cite{Schoen2,Schoen3}.
In particular, Khuri, Marques and Schoen
\cite{Khuri&Marques&Schoen}
proved that the compactness conjecture is true up to dimension $24$, assuming that the positive mass theorem holds.

It turns out that the conjecture is false in higher dimension.
Brendle \cite{Brendle} constructed Riemannian manifolds $(M, g)$ such that the set of constant scalar curvature metrics in
the conformal class of $g$ is non-compact, when the dimension $n\geq 52$.
Modifying the arguments in \cite{Brendle}, Brendle and Marques \cite{Brendle&Marques}
were able to construct such Riemannian manifold when the dimension $25\leq n\leq 52$.

Now suppose that $(M,g)$ is an $n$-dimensional
compact Riemannian manifold with boundary $\partial M$.
The Yamabe problem can also be formulated on Riemannian manifolds with boundary.
And there are two types:\\
(I) Find $\tilde{g}$ conformal to $g$ such that the scalar curvature of $\tilde{g}$ is constant in $M$
and the mean curvature of $\tilde{g}$ is zero on $\partial M$.\\
(II) Find $\tilde{g}$ conformal to $g$ such that the scalar curvature of $\tilde{g}$ is zero in $M$
and the mean curvature of $\tilde{g}$ is constant on $\partial M$.\\
If we write $\tilde{g}=u^{\frac{4}{n-2}}g$ with $0<u\in C^\infty(M)$, then
\begin{equation}\label{0.2}
\begin{split}
-\frac{4(n-1)}{n-2}\Delta_{g}u+R_{g}u &=R_{\tilde{g}} u^{\frac{n+2}{n-2}}~~\mbox{ in }M,\\
\frac{2}{n-2}\frac{\partial u}{\partial\nu_{g}}+H_{g}u &=H_{\tilde{g}} u^{\frac{n}{n-2}}~~\mbox{ on }\partial M.
\end{split}
\end{equation}
Here $\displaystyle\frac{\partial}{\partial\nu_{g}}$ is the outward normal derivative with respect to $g$,
and $H_{g}$ (and $H_{\tilde{g}}$ respectively) is the mean curvature of $g$ (and $\tilde{g}$ respectively).
Therefore, the Yamabe problem with boundary (I) is equivalent to solving (\ref{0.2})
with $R_{\tilde{g}}\equiv c$ and $H_{\tilde{g}}\equiv 0$, i.e.
\begin{equation}\label{0.3}
\begin{split}
-\frac{4(n-1)}{n-2}\Delta_{g}u+R_{g}u&=c u^{\frac{n+2}{n-2}}~~\mbox{ in }M,\\
\frac{2}{n-2}\frac{\partial u}{\partial\nu_{g}}+H_{g}u &=0~~\mbox{ on }\partial M,
\end{split}
\end{equation}
and the Yamabe problem with boundary (II) is equivalent to solving (\ref{0.2})
with $R_{\tilde{g}}\equiv 0$ and $H_{\tilde{g}}\equiv c$, i.e.
\begin{equation}\label{0.4}
\begin{split}
-\frac{4(n-1)}{n-2}\Delta_{g}u+R_{g}u &=0~~\mbox{ in }M,\\
\frac{2}{n-2}\frac{\partial u}{\partial\nu_{g}}+H_{g}u &=c u^{\frac{n}{n-2}}~~\mbox{ on }\partial M.
\end{split}
\end{equation}
The Yamabe problem with boundary (I) and (II) have been studied intensively.
See \cite{Almaraz,Brendle&Chen,Cherrier,Escobar1,Escobar2,Mayer&Ndiaye} and the references therein.

Inspired by the compactness conjecture in the closed case, one can ask
if the set of all solutions to (\ref{0.3}) (and to (\ref{0.4}) respectively)
is compact in $C^2$-topology.
This was studied in \cite{Almaraz1,Almaraz,Disconzi&Khuri,Felli&Ould,Felli&Ould1,Ghimenti&Micheletti,Kim&Musso&Wei}

More generally, one can try to find a metric $\tilde{g}$ conformal to $g$
such that the scalar curvature of $\tilde{g}$ is equal to $a$ in $M$
and the mean curvature of $\tilde{g}$ is equal to $b$ on $\partial M$, for some constants $a$ and $b$.
This includes the Yamabe problem with boundary (I) and (II) as special cases.
In view of (\ref{0.2}), it is equivalent to solving the following:
\begin{equation}\label{0.5}
\begin{split}
-\Delta_{g}u+\frac{n-2}{4(n-1)}R_{g}u  &=c_1 u^{\frac{n+2}{n-2}}~~\mbox{ in }M,\\
\frac{\partial u}{\partial\nu_{g}}+\frac{n-2}{2}H_{g}u &=c u^{\frac{n}{n-2}}~~\mbox{ on }\partial M
\end{split}
\end{equation}
where $c_1=\displaystyle\frac{n-2}{4(n-1)}a$ and $c=\displaystyle\frac{n-2}{2}b$.
This problem of finding conformal metric of constant scalar curvature and constant boundary mean curvature
has been studied in \cite{Chen&Ho&Sun,Chen&Ruan&Sun,Chen&Sun,Han&Li1,Han&Li2}.

Similar to the Yamabe problem on closed manifolds
and the Yamabe problem with boundary (I) and (II),
it would be interesting to study the compactness and non-compactness of (\ref{0.5}).
Almaraz and Wang \cite{Almaraz&Wang}
obtained a compactness result of (\ref{0.5}) when the dimension $n=3$.
Chen and Wu \cite{Chen&Wu} were able to construct
Riemannian manifolds $(M,g)$ with boundary $\partial M$
such that the set of all solutions to (\ref{0.5}) is non-compact,
when the dimension $n\geq 62$.
In order to state their result, we need some notations.
We define the functional
\begin{equation}\label{0.6}
\begin{split}
I_{(M,g)}[u]
&=\int_M\left(\frac{4(n-1)}{n-2}|\nabla_g u|^2+R_gu^2\right)d\mu_g
+2(n-1)\int_{\partial M}H_g u^2 d\sigma_g\\
&\hspace{4mm}-\frac{4(n-1)}{n}c_1\int_M u_+^{\frac{2n}{n-2}}d\mu_g
-4c\int_{\partial M} u_+^{\frac{2(n-1)}{n-2}}d\sigma_g
\end{split}
\end{equation}
for any $u\in H^1(M,g)$,
where $u_+=\max\{u,0\}$.
Let $W$ be a single bubble in $(\mathbb{R}^n_+,|dx|^2)$, i.e.
$$W(x)=\left(\frac{n(n-2)}{c_1}\right)^{\frac{n-2}{4}}
(1+|x-T_c \mathbf{e}_n|^2)^{\frac{2-n}{2}}~~\mbox{ with }
T_c=-\frac{c}{n-2}\sqrt{\frac{n(n-2)}{c_1}},$$
and $S_c$ the energy of $W$, i.e.
\begin{equation*}
\begin{split}
S_c&=I_{(\mathbb{R}^n_+,|dx|^2)}[W]\\
&=\int_{\mathbb{R}^n_+}\frac{4(n-1)}{n-2}|\nabla W|^2dx
-\frac{4(n-1)}{n}c_1\int_{\mathbb{R}^n_+}W^{\frac{2n}{n-2}}dx
-4c\int_{\partial\mathbb{R}^n_+}W^{\frac{2(n-1)}{n-2}} d\sigma\\
&=\frac{4}{n-2}\int_{\mathbb{R}^n_+} |\nabla W|^2dx
+\frac{4c_1}{n-2}\int_{\mathbb{R}^n_+}W^{\frac{2n}{n-2}}dx>0.
\end{split}
\end{equation*}
The following theorem was proved in \cite[Theorem 1.1]{Chen&Wu}.

\begin{theorem}\label{thm1.1}
Let $c\in\mathbb{R}_+$ and $n\geq 62$,
there exists a metric $g$ on $\mathbb{S}^n_+$
such that the Yamabe constant with boundary
 $Y(\mathbb{S}^n_+,\partial \mathbb{S}^n_+, g)>0$,
 and a sequence of positive smooth functions $\{u_\nu\}_{\nu\in\mathbb{N}}$
 with the following properties:\\
 (i) $g$ is not locally conformally flat,\\
 (ii) $\partial\mathbb{S}^n_+$ is umbilic with respect to $g$,\\
 (iii) $u_\nu$ is a positive smooth solution to \eqref{0.5}
 with $c_1=n(n-2)$ and $M=\mathbb{S}^n_+$,\\
 (iv) $I_{(\mathbb{S}^n_+,g)}[u_\nu]<S_c$, and\\
 (v) $\sup_{\partial\mathbb{S}^n_+} u_{\nu}\to\infty$ as $\nu\to\infty$.
\end{theorem}

In this paper, by following and modifying the arguments of Chen and Wu in \cite{Chen&Wu},
we can lower the dimension further to $n\geq 35$.
The following is our main theorem.

\begin{theorem}\label{main}
Let $n\geq 35$.
There exists a constant $\overline{c}\in\mathbb{R}^+$ depending only on $n$,
such that for all $c\in \mathbb{R}^+$ with $c<\overline{c}$,
there exist a metric $g$ on $\mathbb{S}^n_+$
whose Yamabe constant with boundary
 $Y(\mathbb{S}^n_+,\partial \mathbb{S}^n_+, g)>0$,
 and a sequence of positive smooth functions $\{u_\nu\}_{\nu\in\mathbb{N}}$
 with  properties (i)-(v) in Theorem \ref{thm1.1}.
\end{theorem}

We would like to emphasize the main difference between Theorem \ref{thm1.1}
and Theorem \ref{main}.  Chen and Wu were able to prove
that Theorem \ref{thm1.1} holds for all $c\in\mathbb{R}^+$,
while we are only able to prove
that Theorem \ref{main} holds for all $c\in\mathbb{R}^+$ less than
a particular value $\overline{c}$.

In Section \ref{section2}, we will summarize the argument of Chen and Wu
used in the proof of Theorem \ref{thm1.1}.
In Section \ref{section3},
we will then point out the exact place we need to modify
the argument of Chen and Wu in order to prove Theorem \ref{main}.

\section{Summary of the proof of Theorem \ref{thm1.1}}\label{section2}

In this section, we provide a concise exposition of the principal arguments employed by Chen and Wu in the proof of Theorem \ref{thm1.1}. This overview serves to clarify the framework within which the subsequent modifications will be introduced.

From now on, we let $c_1=n(n-2)$
and $T_c=-c/(n-2)$, where $c$ is given in (\ref{0.5}).
Given a pair $(\xi,\epsilon)\in\mathbb{R}^{n-1}\times (0,\infty)$,
we define
$$u_{(\xi,\epsilon)}(x)
=\left(\frac{\epsilon}{\epsilon^2+(x_n-T_c\epsilon)^2
+|x'-\xi|^2}\right)^{\frac{n-2}{2}},$$
where $x=(x',x_n)\in\mathbb{R}^n_+$.
Define
\begin{equation*}
\begin{split}
u_{(\xi,\epsilon,i)} &
=\left(\frac{\epsilon}{\epsilon^2+(x_n-T_c\epsilon)^2
+|x'-\xi|^2}\right)^{\frac{n+2}{2}}
\frac{2\epsilon(x_i-\xi_i)}{\epsilon^2+(x_n-T_c\epsilon)^2
+|x'-\xi|^2},\\
\hat{u}_{(\xi,\epsilon,i)} &
=\left(\frac{\epsilon}{\epsilon^2+(x_n-T_c\epsilon)^2
+|x'-\xi|^2}\right)^{\frac{n}{2}}
\frac{2\epsilon(x_i-\xi_i)}{\epsilon^2+(x_n-T_c\epsilon)^2
+|x'-\xi|^2}
\end{split}
\end{equation*}
for $1\leq i\leq n-1$, and
\begin{equation*}
\begin{split}
u_{(\xi,\epsilon,n)} &
=\left(\frac{\epsilon}{\epsilon^2+(x_n-T_c\epsilon)^2
+|x'-\xi|^2}\right)^{\frac{n+2}{2}}
\frac{(1+T_c^2)\epsilon^2-x_n^2-|x'-\xi|^2}{\epsilon^2+(x_n-T_c\epsilon)^2
+|x'-\xi|^2},\\
\hat{u}_{(\xi,\epsilon,n)} &
=\left(\frac{\epsilon}{\epsilon^2+(x_n-T_c\epsilon)^2
+|x'-\xi|^2}\right)^{\frac{n}{2}}
\frac{(1+T_c^2)\epsilon^2-x_n^2-|x'-\xi|^2}{\epsilon^2+(x_n-T_c\epsilon)^2
+|x'-\xi|^2}.
\end{split}
\end{equation*}
Define
$$\mathcal{E}
=\left\{w\in L^{\frac{2n}{n-2}}(\mathbb{R}^n_+)
\cap  L^{\frac{2(n-1)}{n-2}}(\partial\mathbb{R}^n_+)
\cap H^1_{loc}(\mathbb{R}^n_+)
:\int_{\mathbb{R}^n_+}|\nabla w|^2<\infty\right\}$$
and
$$\mathcal{E}_{(\xi,\epsilon)}
=\left\{w\in\mathcal{E}
:2n\int_{\mathbb{R}^n_+}w u_{(\xi,\epsilon,a)}
-T_c\int_{\partial \mathbb{R}^n_+}
w \hat{u}_{(\xi,\epsilon,a)}=0
\mbox{ for all }1\leq a\leq n\right\}.$$
We define a norm on $\mathcal{E}$ by
$\|w\|_{\mathcal{E}}
=\displaystyle\left(\int_{\mathbb{R}^n_+}|\nabla w|^2\right)^{\frac{1}{2}}$
Clearly, $u_{(\xi,\epsilon)}\in\mathcal{E}_{(\xi,\epsilon)}$.

We introduce a multi-linear form $\overline{W}:\mathbb{R}^{n-1}
\times \mathbb{R}^{n-1}\times \mathbb{R}^{n-1}\times \mathbb{R}^{n-1}
\to\mathbb{R}$
satisfying the same algebraic properties of the Weyl tensor on $\partial\mathbb{R}^n_+$.
Moreover, we assume
$$\sum_{i,j,k,l=1}^{n-1}(\overline{W}_{ikjl}+\overline{W}_{iljk})^2>0.$$
If $x=(x',x_n)\in\mathbb{R}^n_+$, we identify $x'$ with
$(x',0)\in\partial\mathbb{R}^n_+$ and define
\begin{equation}\label{4.1}
 H_{ij}(x)=H_{ij}(x')=\overline{W}_{ikjl}x_kx_l~~\mbox{ and }~~
H_{na}(x)=0,
\end{equation}
as well as
$\overline{H}_{ab}(x)=f(|x'|^2)H_{ab}(x)$,
where $f(s)$ is a polynomial of degree
$d$ for $0\leq d<(n-6)/4$
and to be determined.

We then define
\begin{equation*}
\begin{split}
 \mathcal{F}(\xi,\epsilon) &
=\frac{1}{2}\int_{\mathbb{R}^n_+}\sum_{i,j,l=1}^{n-1}
\overline{H}_{il}\overline{H}_{jl}
\partial_i u_{(\xi,\epsilon)}\partial_j u_{(\xi,\epsilon)}
-\frac{n-2}{16(n-1)}
\int_{\mathbb{R}^n_+}\sum_{i,j,l=1}^{n-1}\sum_{i,j,l=1}^{n-1}
(\partial_l \overline{H}_{ij})^2 u_{(\xi,\epsilon)}^2\\
&\hspace{4mm}+
\int_{\mathbb{R}^n_+}\sum_{i,j=1}^{n-1}
\overline{H}_{ij}\partial_i\partial_j u_{(\xi,\epsilon)} z_{(\xi,\epsilon)}
\end{split}
\end{equation*}
for $(\xi,\epsilon)\in\mathbb{R}^n_+\times (0,\infty)$,
where
$z_{(\xi,\epsilon)}(x)=\mu^{-1}w_{(\xi,\epsilon)}(x)$
satisfying
\begin{equation*}
\begin{split}
&\int_{\mathbb{R}^n_+}\left(\langle\nabla z_{(\xi,\epsilon)},\nabla\varphi\rangle
-n(n+2)u_{(\xi,\epsilon)}^{\frac{4}{n-2}}z_{(\xi,\epsilon)}\varphi\right)
+nT_c\int_{\partial\mathbb{R}^n_+}u_{(\xi,\epsilon)}^{\frac{2}{n-2}}
w_{(\xi,\epsilon)}\varphi\\
&=-\int_{\mathbb{R}^n_+}\mu\lambda^{2d}f(\lambda^{-2}|x'|^2)H_{ij}(x)\partial_i\partial_j
u_{(\xi,\epsilon)}\varphi
\end{split}
\end{equation*}
for all test function $\varphi\in\mathcal{E}_{(\xi,\epsilon)}$.

If $f(s)=\displaystyle\sum_{i=0}^d a_i s^i$,
$a_i\in\mathbb{R}$ for $1\leq i\leq d$,
then $\alpha_q$ are constants defined as
\begin{equation}\label{5.3}
\sum_{q=0}^{2d}\alpha_qs^q:=(n+1)f(s)^2+4sf(s)f'(s)+2s^2f'(s)^2.
\end{equation}
Also, we define
\begin{equation}\label{5.4}
c_q=\int_0^\infty(1+(t-T_c)^2)^{\frac{5+2q-n}{2}}dt,~~q\in\mathbb{N}\mbox{ and }0\leq q\leq 2d.
\end{equation}
Then we have
\begin{equation}\label{5.6}
I(s)=\sum_{q=0}^{2d}\left(c_q\alpha_q s^{q+2}\prod_{j=0}^q\frac{n-1+2j}{n-5-2j}\right).
\end{equation}
Let
\begin{equation}\label{5.7}
2f(s)f'(s)+sf'(s):=\sum_{q=0}^{2d-1}\beta_q s^q.
\end{equation}
Then
\begin{equation}\label{5.8}
J(s)=\sum_{q=0}^{2q-1}\left(c_q\beta_q s^{q+2}\prod_{j=0}^q\frac{n+3+2j}{n-5-2j}\right).
\end{equation}
In order to show that $\mathcal{F}(\xi,\epsilon)$
has a strict local minimum at $(0,1)$,
Chen and Wu showed that
it suffices to
show that
$I(1)>0$, $I'(1)=0$, $I''(1)<0$ and $J(1)<0$.
By choosing $d=1$, i.e. $f$ is a polynomial of degree $1$,
Chen and Wu was able to show this when $n\geq 62$.

\begin{prop}[Proposition 5.9 in \cite{Chen&Wu}]
\label{Chen_prop5.9}
Let $n\geq 62$. There exists a polynomial
$f(s)=-s+a_0$ such that
$I(1)>0$, $I'(1)=0$, $I''(1)<0$ and $J(1)<0$.
This implies that $\mathcal{F}(\xi,\epsilon)$
has a strict local minimum at the point $(0,1)$.
\end{prop}

Using Proposition \ref{Chen_prop5.9}, Chen
and Wu proved the following:

\begin{prop}[Proposition 6.1 in \cite{Chen&Wu}]\label{Chen_prop6.1}
For $T_c\in\mathbb{R}_-$ and $n\geq n_0$, let $g=\exp(h)$ be
a smooth Riemannian metric in $\overline{\mathbb{R}^n_+}$,
where $h$ is a symmetric trace-free two tensor in $\overline{\mathbb{R}^n_+}$
satisfying
$$\left\{
  \begin{array}{ll}
    h_{ij}(x)=\mu\lambda^2 f(\lambda^{-2}|x'|^2)H_{ij}(x), & \hbox{in $B_\rho^+$;} \\
    h_{ab}(x)=0, & \hbox{in $\overline{\mathbb{R}^n_+}\setminus B_\rho^+$;} \\
    h_{na}(x)=0, & \hbox{in $\overline{\mathbb{R}^n_+}$,}
  \end{array}
\right.$$
where $0<\mu\leq 1$, $0<\lambda\leq\rho\leq 1$,
$1\leq i\leq n-1$, $1\leq a, b\leq n$ and $H_{ab}$ is defined
as in \eqref{4.1}. Assume that
$|h(x)|+|\partial h(x)|+|\partial^2 h(x)|\leq \alpha$
for all $x\in \overline{\mathbb{R}^n_+}$.
If $\alpha$ and $\mu^{-2}\lambda^{n-10}\rho^{2-n}$
are sufficiently small, then there exists a positive
smooth solution of
\begin{equation}\label{ChenWu_eq6.1}
\left\{
  \begin{array}{ll}
    \displaystyle\Delta_g v-\frac{n-1}{4(n-2)}R_g v+n(n-2)v^{\frac{n+2}{n-2}}=0, & \hbox{in $\mathbb{R}^n_+$;} \\
    \displaystyle\frac{\partial v}{\partial x_n}=(n-2)T_c v^{\frac{n}{n-2}}, & \hbox{on $\partial\mathbb{R}^n_+$.}
  \end{array}
\right.
\end{equation}
Moreover, there exists $C=C(n,T_c)>0$ such that
$$\sup_{B_\lambda^+(0)} v\geq C\lambda^{\frac{2-n}{2}}$$
and
\begin{equation*}
\begin{split}
&2(n-2)\int_{\mathbb{R}^n_+}v^{\frac{2n}{n-2}}-\frac{n-2}{n-1}T_c\int_{\partial\mathbb{R}^n_+}v^{\frac{2(n-1)}{n-2}}\\
&<2(n-2)\int_{\mathbb{R}^n_+}u_{(0,1)}^{\frac{2n}{n-2}}-\frac{n-2}{n-1}T_c\int_{\partial\mathbb{R}^n_+}u_{(0,1)}^{\frac{2(n-1)}{n-2}}.
\end{split}
\end{equation*}
\end{prop}

\begin{proof}[\textbf{Sketch of the Proof}]
  For the sake of completeness, we provide the proof here. However, this is only a sketch, and for the full proof, the reader is referred to \cite{Chen&Wu}. 
  
  \textbf{Step 1.} If $h$ satisfies the conditions of Proposition \ref{Chen_prop6.1} and $g=\exp(h)$, then we can establish the following: Given any pair $(\xi,\epsilon)\in\mathbb{R}^{n-1}\times(0,\infty)$, there exists a unique function $v_{(\xi,\epsilon)}\in\mathcal{E}$ such that $v_{(\xi,\epsilon)}-u_{(\xi,\epsilon)}\in \mathcal{E}_{(\xi,\epsilon)}$, and 
  \begin{equation*}
    \begin{split}
      \int_{\mathbb{R}^n_+}&\left( \langle \nabla v_{(\xi,\epsilon)},\nabla \varphi\rangle_g,c_nR_gv_{(\xi,\epsilon)}\varphi-n(n-2)|v_{(\xi,\epsilon)}|^\frac{4}{n-2}v_{(\xi,\epsilon)}\varphi\right)\\
      +&\int_{\partial\mathbb{R}^n_+}\left(d_nh_gv_{(\xi,\epsilon)}+(n-2)T_c|v_{(\xi,\epsilon)}^\frac{2}{n-2}v_{(\xi,\epsilon)}\right)\varphi=0    \end{split}
  \end{equation*}
  for all $\varphi\in\mathcal{E}_{(\xi,\epsilon)}$. Moreover, there exists a positive constant $C$, depending only on $T_c$ and $n$ such that 
  \begin{equation}\label{ChenWu_eq3.11}
    \begin{split}
      \|&v_{(\xi,\epsilon)}-u_{(\xi,\epsilon)}\|_{\mathcal{E}}\\
      &\leq C\|\Delta_gu_{(\xi,\epsilon)}-c_nR_gu_{(\xi,\epsilon)}+n(n-2)u_{(\xi,\epsilon)}^\frac{n+2}{n-2}\|_{L^\frac{2n}{n+2}(\mathbb{R}^n_+)}+C\|h_gu_{(\xi,\epsilon)}\|_{L^\frac{2(n-1)}{n}(\partial \mathbb{R}^n_+)}.
    \end{split}
  \end{equation}
  
  \textbf{Step 2.} Given a pair $(\xi,\epsilon)\in\mathbb{R}^{n-1}\times(0,\infty)$, we define the following energy functional
  \begin{equation*}
    \begin{split}
      &\mathcal{F}_g(\xi,\epsilon)\\
      :=&\int_{\mathbb{R}^n_+}\left(|\nabla v_{(\xi,\epsilon)}|^2_g+c_nR_gv_{(\xi,\epsilon)}^2-(n-2)^2|v_{(\xi,\epsilon)}^\frac{2n}{n-2}\right)+d_n\int_{\partial\mathbb{R}^n_+}h_gv_{(\xi,\epsilon)}^2\\
      &+\frac{(n-2)^2}{n-1}T_c\int_{\partial\mathbb{R}^n_+}|v_{(\xi,\epsilon)}|^\frac{2(n-1)}{n-2}-2(n-2)\int_{\mathbb{R}^n_+}u^\frac{2n}{n-2}_{(\xi,\epsilon)}+\frac{n-2}{n-1}T_c\int_{\partial \mathbb{R}^n_+}u^\frac{2(n-1)}{n-2}_{(\xi,\epsilon)}. 
    \end{split}
  \end{equation*}
  Then the functional $\mathcal{F}_g$ is continuously differentiable. Moreover, if $(\xi,\epsilon)$ is a critical point of $\mathcal{F}_g$, then the function $v_{(\xi,\epsilon)}$ is a positive smooth solution of 
  \begin{equation*}
    \begin{split}
      \left\{
        \begin{array}{ll}
          -\Delta_gv_{(\xi,\epsilon)}+c_nR_gv_{(\xi,\epsilon)}=n(n-2)v_{(\xi,\epsilon)}^\frac{n+2}{n-2}, & \hbox{in $\mathbb{R}^n_+$,} \\
          \frac{\partial v_{(\xi,\epsilon)}}{\partial x_n}-d_n h_gv_{(\xi,\epsilon)}=(n-2)T_cv_{(\xi,\epsilon)}^\frac{n}{n-2}, & \hbox{on $\partial \mathbb{R}^n_+$.}
        \end{array}
      \right.
    \end{split}
  \end{equation*}
  
\textbf{Step 3.} Applying Step 1 to each pair $(\xi,\epsilon)\in\mathbb{R}^{n-1}\times(0,\infty)$, we choose $v_{(\xi,\epsilon)}$ to be the unique element of $\mathcal{E}$ such that $v_{(\xi,\epsilon)}-u_{(\xi,\epsilon)}\in\mathcal{E}_{(\xi,\epsilon)}$ and 
\begin{equation*}
  \begin{split}
 & \int_{\mathbb{R}^n_+}\left(\langle \nabla v_{(\xi,\epsilon)},\nabla \varphi\rangle_g+c_nR_gv_{(\xi,\epsilon)}\varphi-n(n-2)|v_{(\xi,\epsilon)}|^\frac{4}{n-2}v_{(\xi,\epsilon)}\varphi\right)\\
&+(n-2)T_c\int_{\partial \mathbb{R}^n_+}|v_{(\xi,\epsilon)}|^\frac{2}{n-2}v_{(\xi,\epsilon)}\varphi=0
\end{split}
\end{equation*}
  for all $\varphi\in \mathcal{E}_{(\xi,\epsilon)}$. Let $\Omega=\left\{(\xi,\epsilon)\in\mathbb{R}^{n-1}\times(0,\infty):|\xi|<1,\frac{1}{2}<\epsilon<1\right\}$. We define the function $w_{(\xi,\epsilon)}$ as the unique element of $\mathcal{E}_{(\xi,\epsilon)}$ satisfying
\begin{equation*}
  \begin{split}
  \int_{\mathbb{R}^n_+}&\left(\langle \nabla w_{(\xi,\epsilon)},\nabla \varphi\rangle-n(n+2)u_{(\xi,\epsilon)}^\frac{4}{n-2}w_{(\xi,\epsilon)}\varphi\right)+n T_c\int_{\partial \mathbb{R}_+^n}u^\frac{2}{n-2}_{(\xi,\epsilon)}w_{(\xi,\epsilon)}\varphi\\
=&-\int_{\mathbb{R}^n_+}\mu \lambda^{2d}f(\lambda^{-2}|x'|^2H_{ij}(x)\partial_i\partial_ju_{(\xi,\epsilon)}\varphi
\end{split}
\end{equation*}
  for all $\varphi\in\mathcal{E}_{(\xi,\epsilon)}$. Then we can establish the following: For any $(\xi,\epsilon)\in\lambda\Omega$, there holds
\begin{equation*}
  \begin{split}
 &\left|\mathcal{F}_g(\xi,\epsilon)-\frac{1}{2}\int_{B_\rho^+(0)}\sum_{a,b,c=1}^nh_{ac}h_{bc}\partial_au_{(\xi,\epsilon)}\partial_bu_{(\xi,\epsilon)}+\frac{c_n}{4}\int_{B_\rho^+(0)}\sum_{a,b,c=1}^n(\partial_ch_{ab})^2u_{(\xi,\epsilon)}^2\right.\\
 &\left.-\mu\lambda^{2d}\int_{\mathbb{R}^n_+}f(\lambda^{-2}|x'|^2)w_{(\xi,\epsilon)}\sum_{a,b=1}^nH_{ab}(x)\partial_a\partial_bu_{(\xi,\epsilon)}\right|\\
&\leq C\mu^\frac{2(n-1)}{n-2}\lambda^\frac{(4d+4)(n-1)}{n-2}+C\mu\lambda^{2d+2+\frac{n-2}{2}}\rho^\frac{2-n}{2}+C\lambda^{n-2}\rho^{2-n}.
\end{split}
\end{equation*}

\textbf{Step 4.} It follows from Proposition \ref{Chen_prop5.9} that $(0,1)$ is a strict local minimum point of $\mathcal{F}(\xi,\epsilon)$. Hence, we can find an open set $\Omega'\subset \Omega$ such that $(0,1)\in \Omega'$ and $\mathcal{F}(0,1)<\inf_{(\xi,\epsilon)\in \partial \Omega'}\mathcal{F}(\xi,\epsilon)<0$. By Step 3 with $d=1$, we have
\begin{equation*}
  \begin{split}
    &\left|\mathcal{F}_g(\lambda\xi,\lambda\epsilon)-\lambda^{8}\mu^2\mathcal{F}(\xi,\epsilon)\right|\\
    &\leq C\mu^\frac{2(n-1)}{n-2}\lambda^\frac{8(n-1)}{n-2}+C\mu\lambda^\frac{n+6}{2}\rho^\frac{2-n}{2}+C\lambda^{n-2}\rho^{2-n}
  \end{split}
\end{equation*}
for all $(\xi,\epsilon)\in\Omega$, equivalently,
\begin{equation*}
  \begin{split}
    &\left|\lambda^{-8}\mu^{-2}\mathcal{F}_g(\lambda\xi,\lambda\epsilon)-\mathcal{F}(\xi,\epsilon)\right|\\
    &\leq C\mu^\frac{2}{n-2}\lambda^\frac{8}{n-2}+C\mu^{-1}\lambda^\frac{n-10}{2}\rho^\frac{2-n}{2}+C\mu^{-2}\lambda^{n-10}\rho^{2-n}
  \end{split}
\end{equation*}
for all $(\xi,\epsilon)\in \Omega$. If $\mu^{-2}\lambda^{n-10}\rho^{2-n}$ is sufficiently small, then we have
$$\mathcal{F}_g(0,\lambda)<\int_{(\xi,\epsilon)\in\partial \Omega'}\mathcal{F}_g(\lambda\xi,\lambda\epsilon)<0.$$
Consequently, there exists $(\overline{\xi},\overline{\epsilon})\in\Omega'$ such that 
$$\mathcal{F}_g(\lambda\overline{\xi},\lambda\overline{\epsilon})=\inf_{(\xi,\epsilon)\in\Omega'}\mathcal{F}_g(\lambda \xi,\lambda \epsilon)<0.$$
It follows from Step 2 that that the function $v=v_{(\lambda\overline{\xi},\lambda\overline{\epsilon})}$ obtained in Step 1 is a positive smooth solution to \eqref{ChenWu_eq6.1}. By definition of $\mathcal{F}_g$ we have
\begin{equation*}
  \begin{split}
    2(n-2)&\int_{\mathbb{R}^n_+}v^\frac{2n}{n-2}-\frac{n-2}{n-1}T_c\int_{\partial \mathbb{R}^n_+}v^\frac{2(n-1)}{n-2}\\
    =&\mathcal{F}_g(\lambda\overline{\xi},\lambda\overline{\epsilon})+2(n-2)\int_{\mathbb{R}^n_+}u^\frac{2n}{n-2}_{(\lambda\overline{\xi},\lambda\overline{\epsilon})}-\frac{n-2}{n-1}T_c\int_{\partial \mathbb{R}^n_+}u^\frac{2(n-1)}{n-2}_{(\lambda\overline{\xi},\lambda\overline{\epsilon})},
  \end{split}
\end{equation*} 
whence
\begin{equation*}
  \begin{split}
    2(n-2)\int_{\mathbb{R}^n_+}v^\frac{2n}{n-2}&-\frac{n-2}{n-1}T_c\int_{\partial \mathbb{R}^n_+}v^\frac{2(n-1)}{n-2}\\
    &<2(n-2)\int_{\mathbb{R}^n_+}u^\frac{2n}{n-2}_{(0,1)}-\frac{n-2}{n-1}T_c\int_{\partial \mathbb{R}^n_+}u^\frac{2(n-1)}{n-2}_{(0,1)}.
  \end{split}
\end{equation*}
Using \eqref{ChenWu_eq3.11}, we estimate
$$\|v-u_{(\lambda\overline{\xi},\lambda\overline{\epsilon})}\|_{L^\frac{2n}{n-2}(B_\lambda^+(0))}\leq\|v-u_{(\lambda\overline{\xi},\lambda\overline{\epsilon})}\|_{L^\frac{2n}{n-2}(\partial \mathbb{R}^n_+)}\leq C\alpha.$$
Then
$$|B_\lambda^+(0)|^\frac{n-2}{2n}\sup_{B_\lambda^+(0)}v\geq\|v\|_{L^\frac{2n}{n-2}(B_\lambda^+(0))}\geq-C\alpha+\|u_{(\lambda\overline{\xi},\lambda\overline{\epsilon})}\|_{L^\frac{2(n-1)}{n-2}(B_\lambda^+(0))}.$$
Hence, if $\alpha$ is sufficiently small, then we obtain
$$\sup_{B_\lambda^+(0)}v\geq C\lambda^\frac{2-n}{2}.$$
This completes the proof.
\end{proof}

We remark that in the original statement of Proposition 6.1
in \cite{Chen&Wu}, it is assumed that $f(s)=-s+a_0$
and $n_0=62$.
However, one can see from the proof of Proposition 6.1 in \cite{Chen&Wu}
that Proposition \ref{Chen_prop6.1} is still true.

With Proposition \ref{Chen_prop6.1},
one can follow the proof of Theorem
6.2 in \cite{Chen&Wu} to prove the following:

\begin{theorem}[Theorem 6.2 in \cite{Chen&Wu}]\label{Chen_thm6.2}
For $T_c\in\mathbb{R}_-$ and $n\geq n_0$, there exists
a smooth Riemannian metric $g$ in $\overline{\mathbb{R}^n_+}$
with the following properties:\\
(i) $g_{ab}(x)=\delta_{ab}$ for $x\in \overline{\mathbb{R}^n_+}\setminus B_{1/2}^+(0)$,\\
(ii) $g$ is not locally conformally flat,\\
(iii) $\partial\mathbb{R}^n_+$ is totally geodesic with respect to $g$,\\
(iv) there exists a sequence of positive smooth functions
$\{v_\nu\}_{\nu\in\mathbb{N}}$ satisfying
\begin{equation*}
\left\{
  \begin{array}{ll}
    \displaystyle\Delta_g v_\nu-\frac{n-1}{4(n-2)}R_g v_\nu+n(n-2)v_\nu^{\frac{n+2}{n-2}}=0, & \hbox{in $\mathbb{R}^n_+$;} \\
    \displaystyle\frac{\partial v_\nu}{\partial x_n}=(n-2)T_c v_\nu^{\frac{n}{n-2}}, & \hbox{on $\partial\mathbb{R}^n_+$}
  \end{array}
\right.
\end{equation*}
for all $\nu$. Moreover, there hold
\begin{equation*}
\begin{split}
&2(n-2)\int_{\mathbb{R}^n_+}v_\nu^{\frac{2n}{n-2}}-\frac{n-2}{n-1}T_c\int_{\partial\mathbb{R}^n_+}v_\nu^{\frac{2(n-1)}{n-2}}\\
&<2(n-2)\int_{\mathbb{R}^n_+}u_{(0,1)}^{\frac{2n}{n-2}}-\frac{n-2}{n-1}T_c\int_{\partial\mathbb{R}^n_+}u_{(0,1)}^{\frac{2(n-1)}{n-2}}.
\end{split}
\end{equation*}
for all $\nu$, i.e.
$I_{(\mathbb{R}^n_+,|dx|^2)}[v_\nu]<S_c$,
and $\sup_{B_1(0)^+} v_\nu\to\infty$ as $\nu\to\infty$.
\end{theorem}

Now Theorem \ref{thm1.1} is a direct consequence of Theorem \ref{Chen_thm6.2}
by taking $n_0=62$.

\section{Proof of Theorem \ref{main}}\label{section3}

In order to lower the dimension further, we are going to choose
$f(s)$ to be a polynomial of degree $6$. In addition, we will establish 
the corresponding version of Proposition \ref{Chen_prop5.9}.
More precisely, we prove in this section the following:

\begin{prop}\label{main_prop}
Let $n\geq 35$. There exists a polynomial
$f(s)=\displaystyle\sum_{i=0}^6 a_i s^i$ such that
$I(1)>0$, $I'(1)=0$, $I''(1)<0$ and $J(1)<0$, provided that $T_c=-\displaystyle\frac{c}{n-2}$ is sufficiently close to $0$.
This implies that $\mathcal{F}(\xi,\epsilon)$
has a strict local minimum at the point $(0,1)$.
\end{prop}

With Proposition \ref{main_prop},
Theorem \ref{main} now follows from Proposition \ref{Chen_prop6.1} and Theorem \ref{Chen_thm6.2}
with $n_0=35$. Thus it boils down to proving Proposition \ref{main_prop}.

Now we choose $d=6$ such that
$f(s)=a_0+a_1s+a_2s^2+a_3s^3+a_4s^4+a_5s^5+a_6s^6$,
where $a_1, a_3,a_5$ are negative constants and  $a_2,a_4,a_6$ are positive constants to be chosen later.
It follows from (\ref{5.3}) that
\begin{equation}\label{a}
  \begin{split}
    &\alpha_0=(n+1)a_0^2,\\
    &\alpha_1=2(n+3)a_0a_1,\\
    &\alpha_2=(n+7)a_1^2+2(n+5)a_0a_2,\\
    &\alpha_3=2(n+11)a_1a_2+2(n+7)a_0a_3,\\
    &\alpha_4=(n+17)a_2^2+2(n+15)a_1a_3+2(n+9)a_0a_4,\\
    &\alpha_5=2(n+23)a_2a_3+2(n+19)a_1a_4+2(n+11)a_0a_5,\\
    &\alpha_6=(n+31)a_3^2+2(n+29)a_2a_4+2(n+23)a_1a_5+2(n+13)a_0a_6,\\
    &\alpha_7=2(n+39)a_3a_4+2(n+35)a_2a_5+2(n+27)a_1a_6,\\
    &\alpha_8=(n+49)a_4^2+2(n+47)a_3a_5+2(n+41)a_2a_6,\\
    &\alpha_9=2(n+59)a_4a_5+2(n+55)a_3a_6,\\
    &\alpha_{10}=(n+71)a_5^2+2(n+69)a_4a_6,\\
    &\alpha_{11}=2(n+83)a_5a_6,\\
    &\alpha_{12}=(n+97)a_6^2.
  \end{split}
\end{equation}
Differentiating (\ref{5.6}) with respect to $s$
yields
\begin{equation}\label{iprime}
\begin{split}
I'(s)&=\sum_{q=0}^{12}\left(c_q\alpha_q(q+2)s^{q+1}\prod_{j=0}^q\frac{n-1+2j}{n-5-2j}\right).
\end{split}
\end{equation}
Combining this with  (\ref{a}), one can see that $I'(1)$ can be considered as a quadratic polynomial in $a_0$. More precisely, $p_n(a_0):=I'(1)=A_na_0^2+B_na_0+C_n$ where
\begin{equation}\label{anbncn}
  \begin{split}
    &A_n=\frac{2(n-1)(n+1)}{n-5}c_0,\\
    &B_n=2(n-1)\sum_{l=1}^6\left[(l+2)a_lc_l\prod_{j=0}^l\frac{n+1+2j}{n-5-2j}\right],\\
    &C_n=4(n+7)a_1^2c_2\prod_{0}^2\frac{n-1+2j}{n-5-2j}+10(n+11)a_1a_2c_3\prod_{j=0}^3\frac{n-1+2j}{n-5-2j}\\
    &\hspace{8mm}+6\left[(n+17)a_2^2+2(n+15)a_1a_3\right]c_4\prod_{j=0}^4\frac{n-1+2j}{n-5-2j}\\
    &\hspace{8mm}+14\left[(n+23)a_2a_3+(n+19)a_1a_4\right]c_5\prod_{j=0}^5\frac{n-1+2j}{n-5-2j}\\
    &\hspace{8mm}+8\left[(n+31)a_3^2+2(n+29)a_2a_4+2(n+23)a_1a_5\right]c_6\prod_{j=0}^6\frac{n-1+2j}{n-5-2j}\\
    &\hspace{8mm}+18\left[(n+39)a_3a_4+(n+35)a_2a_5+(n+27)a_1a_6\right]c_7\prod_{j=0}^7\frac{n-1+2j}{n-5-2j}\\
    &\hspace{8mm}+10\left[(n+49)a_4^2+2(n+47)a_3a_5+2(n+41)a_2a_6\right]c_8\prod_{j=0}^8\frac{n-1+2j}{n-5-2j}\\
    &\hspace{8mm}+22\left[(n+59)a_4a_5+(n+55)a_3a_6\right]c_9\prod_{j=0}^9\frac{n-1+2j}{n-5-2j}\\
    &\hspace{8mm}+12\left[(n+71)a_5^2+2(n+69)a_4a_6\right]c_{10}\prod_{j=0}^{10}\frac{n-1+2j}{n-5-2j}\\
    &\hspace{8mm}+26(n+83)a_5a_6c_{11}\prod_{j=0}^{11}\frac{n-1+2j}{n-5-2j}+14(n+97)a_6^2c_{12}\prod_{j=0}^{12}\frac{n-1+2j}{n-5-2j}.
  \end{split}
\end{equation}
Note that it follows from \eqref{5.4} that  $c_q$ is a continuous function of $T_c$, i.e. $c_q=c_q(T_c)$.
Hence, $A_n=A_n(T_c)$, $B_n=B_n(T_c)$, and $C_n=C_n(T_c)$ can also be regarded as continuous functions of $T_c$. Let $discrim(p_n)(T_c):=B_n(T_c)^2-4A_n(T_c)C_n(T_c)$,
which is the discriminant of $I'(1)$, the quadratic polynomial  in $a_0$.
 We will show that $discrim(p_n)(0)>0$. By continuity, this implies that $discrim(p_n)(T_c)>0$ when $T_c=-\displaystyle\frac{c}{n-2}$ is sufficiently close to $0$. When $T_c=0$,
the explicit value of $c_q(0)$ can be computed by using \eqref{5.4}:
\begin{equation}\label{cq0}
  c_q(0)=\frac{\sqrt{\pi}\Gamma(\frac{n-6-2q}{2})}{2\Gamma(\frac{n-5-2q}{2})},
\end{equation}
where $\Gamma(s)$ denotes the gamma function.

Now we choose  $a_1=-10$, $a_2=10^{-4}$,  $a_3=- 10^{-3}$,  $a_4=1.84\times 10^{-1}$,   $a_5=-2.65\times 10^{-2}$, and  $a_6=7.37\times 10^{-4}$.
By \eqref{anbncn} and \eqref{cq0}, a direct computation with the help of \texttt{Mathematica} shows that $discrim(p_n)(0)>0$ for $35\leq n\leq 62$ (c.f. \cite{code}).
 This tells us that there exists  $a_0\in\mathbb{R}$ such that $I'(1)=0$,  when $T_c=-\displaystyle\frac{c}{n-2}$ is sufficiently close to $0$.

Differentiating (\ref{iprime}) with respect to $s$
yields
\begin{equation*}
  I''(s)=\sum_{q=0}^{12}\left(c_q\alpha_q(q+2)(q+1)s^{q}\prod_{j=0}^q\frac{n-1+2j}{n-5-2j}\right).
\end{equation*}
Combining this with  (\ref{a}), we have the following:
\begin{equation*}
  \begin{split}
    I''(1)&=\frac{2(n-1)(n+1)}{n-5}a_0^2c_0+2(n-1)a_0\sum_{l=1}^6\left[(l+1)(l+2)a_lc_l\prod_{j=0}^l\frac{n+1+2j}{n-5-2j}\right]\\
    &\hspace{4mm}+12(n+7)a_1^2c_2\prod_{j=0}^2\frac{n-1+2j}{n-5-2j}+40(n+11)a_1a_2c_3\prod_{j=0}^3\frac{n-1+2j}{n-5-2j}\\
    &\hspace{4mm}+30\left[(n+17)a_2^2+2(n+15)a_1a_3\right]c_4\prod_{j=0}^4\frac{n-1+2j}{n-5-2j}\\
    &\hspace{4mm}+84\left[(n+23)a_2a_3+(n+19)a_1a_4\right]c_5\prod_{j=0}^5\frac{n-1+2j}{n-5-2j}\\
    &\hspace{4mm}+56\left[(n+31)a_3^2+2(n+29)a_2a_4+2(n+23)a_1a_5\right]c_6\prod_{j=0}^6\frac{n-1+2j}{n-5-2j}\\
    &\hspace{4mm}+144\left[(n+39)a_3a_4+(n+35)a_2a_5+(n+27)a_1a_6\right]c_7\prod_{j=0}^7\frac{n-1+2j}{n-5-2j}\\
    &\hspace{4mm}+90\left[(n+49)a_4^2+2(n+47)a_3a_5+2(n+41)a_2a_6\right]c_8\prod_{j=0}^8\frac{n-1+2j}{n-5-2j}\\
    &\hspace{4mm}+220\left[(n+59)a_4a_5+(n+55)a_3a_6\right]c_9\prod_{j=0}^9\frac{n-1+2j}{n-5-2j}\\
    &\hspace{4mm}+132\left[(n+71)a_5^2+2(n+69)a_4a_6\right]c_{10}\prod_{j=0}^{10}\frac{n-1+2j}{n-5-2j}\\
    &\hspace{4mm}+312(n+83)a_5a_6c_{11}\prod_{j=0}^{11}\frac{n-1+2j}{n-5-2j}+182(n+97)a_6^2c_{12}\prod_{j=0}^{12}\frac{n-1+2j}{n-5-2j}.
  \end{split}
\end{equation*}
Note that $I''(1)=I''(1)(T_c)$ can also be regarded as a continuous function of $T_c$. We will show that $I''(1)(0)<0$, which implies that $I''(1)(T_c)<0$
when $T_c=-\displaystyle\frac{c}{n-2}$ is close to $0$. Let $a_0=a_0(n)$ be the largest root of the quadratic polynomial $p_n$ when $T_c=0$, i.e.
\begin{equation*}
  a_0(n)=\frac{-B_n(0)+\sqrt{B_n(0)^2-4A_n(0)C_n(0)}}{2A_n(0)}.
\end{equation*}
We remark that, unlike other $a_i$'s that have a fixed value, $a_0$ depends on $n$.
For example,  we have
\begin{equation*}
  a_0(35)=\frac{388694358052+\sqrt{31382012570285343694}}{21443906250}
\end{equation*}
when $n=35$, while
\begin{equation*}
  a_0(62)=\frac{46081869321940760+\sqrt{76861908977503143130771763741035}}{2750533632000000}
\end{equation*}
when $n=62$.
Then direct computations with the help of \texttt{Mathematica} show that $I''(1)(0)<0$ for $35\leq n\leq 62$ (c.f. \cite{code}),
which by continuity implies that $I''(1)(T_c)<0$ for $35\leq n\leq 62$
when $T_c$ is  sufficiently close to $0$.
Likewise, with the help of \texttt{Mathematica},
we can show that $I(1) > 0$  when $T_c=0$ (c.f. \cite{code}), which implies that $I(1) > 0$ when $T_c$ is  sufficiently close to $0$.

It remains  to show that $J(1) < 0$ when $T_c$ is  sufficiently close to $0$. It follows from \eqref{5.7} that
\begin{equation*}
  \begin{split}
    \beta_0=&2a_0a_1,\\
    \beta_1=&a_1+2a_1^2+4a_0a_2,\\
    \beta_2=&2a_2+6a_1a_2+6a_0a_3,\\
    \beta_3=&4a_2^2+3a_3+8a_1a_3+8a_0a_4,\\
    \beta_4=&10a_2a_3+4a_4+10a_1a_4+10a_0a_5,\\
    \beta_5=&6a_3^2+12a_2a_4+5a_5+12a_1a_5+12a_0a_6,\\
    \beta_6=&14a_3a_4+14a_2a_5+6a_6+14a_1a_6,\\
    \beta_7=&8a_4^2+16a_3a_5+16a_2a_6,\\
    \beta_8=&18a_4a_5+18a_3a_6,\\
    \beta_9=&10a_5^2+20a_4a_6,\\
    \beta_{10}=&20a_5a_6,\\
    \beta_{11}=&12a_6^2.
  \end{split}
\end{equation*}
Combining this with \eqref{5.8}, we have
\begin{equation*}
  \begin{split}
    J(1)&=2a_1a_0c_0\prod_{j=0}^0\frac{n+3+2j}{n-5-2j}+(a_1+2a_1^2+4a_2a_0)c_1\prod_{j=0}^1\frac{n+3+2j}{n-5-2j}\\
    &\hspace{4mm}+(2a_2+6a_1a_2+6a_3a_0)c_2\prod_{j=0}^2\frac{n+3+2j}{n-5-2j}\\
    &\hspace{4mm}+(4a_2^2+3a_3+8a_1a_3+8a_4a_0)c_3\prod_{j=0}^3\frac{n+3+2j}{n-5-2j}\\
    &\hspace{4mm}+(10a_2a_3+4a_4+10a_1a_4+10a_5a_0)c_4\prod_{j=0}^4\frac{n+3+2j}{n-5-2j}\\
    &\hspace{4mm}+(6a_3^2+12a_2a_4+5a_5+12a_1a_5+12a_6a_0)c_5\prod_{j=0}^5\frac{n+3+2j}{n-5-2j}\\
    &\hspace{4mm}+(14a_3a_4+14a_2a_5+6a_6+14a_1a_0)c_6\prod_{j=0}^6\frac{n+3+2j}{n-5-2j}\\
    &\hspace{4mm}+(8a_4^2+16a_3a_5+16a_2a_6)c_7\prod_{j=0}^7\frac{n+3+2j}{n-5-2j}\\
    &\hspace{4mm}+(18a_4a_5+18a_3a_6)c_8\prod_{j=0}^8\frac{n+3+2j}{n-5-2j}\\
    &\hspace{4mm}+(10a_5^2+20a_4a_6)c_9\prod_{j=0}^9\frac{n+3+2j}{n-5-2j}\\
    &\hspace{4mm}+22a_5a_6c_{10}\prod_{j=0}^{10}\frac{n+3+2j}{n-5-2j}+12a_6^2c_{11}\prod_{j=0}^{11}\frac{n+3+2j}{n-5-2j}.
  \end{split}
\end{equation*}
Viewing $J(1)$ as a continuous function of $T_c$, with our choice of $a_i$, we can show that $J(1)(0)<0$  for $35\leq n \leq 62$ with the help of \texttt{Mathematica}
(c.f. \cite{code}).
This implies that, by continuity, $J(1)<0$ when $T_c$ is sufficiently close to $0$.
This finishes the proof of Proposition \ref{main_prop}.

\section*{Acknowledgement}

The first author was supported by the National Science and Technology Council (NSTC), Taiwan, with grant Number: 112-2115-M-032-006-MY2.

\bibliographystyle{amsplain}

\begin{thebibliography}{30}

\bibitem{Almaraz1}
S. Almaraz, A compactness theorem for scalar-flat metrics on manifolds with boundary. \textit{Calc. Var. Partial Differential Equations} \textbf{41} (2011), no. 3-4, 341–386

\bibitem{Almaraz2}
S. Almaraz,   Blow-up phenomena for scalar-flat metrics on manifolds with boundary. \textit{J. Differential Equations} \textbf{251} (2011), no. 7, 1813–1840.


\bibitem{Almaraz}
S. Almaraz,
An existence theorem of conformal scalar-flat metrics on manifolds with boundary.
\textit{Pacific J. Math.} \textbf{248} (2010), no. 1, 1–22.


\bibitem{Almaraz&deQueiroz&Wang}
S. Almaraz, O.  de Queiroz, and S. Wang,
A compactness theorem for scalar-flat metrics on 3-manifolds with boundary.
\textit{J. Funct. Anal.} \textbf{277} (2019), no. 7, 2092–2116.


\bibitem{Almaraz&Wang}
S. Almaraz and S. Wang,
A compactness theorem for conformal metrics with constant scalar curvature and constant boundary mean curvature in dimension three.
(2023), preprint. https://arxiv.org/abs/2306.07088


\bibitem{Aubin} T. Aubin,
\'{E}quations diff\'{e}rentielles non lin\'{e}aires et probl\`{e}me de Yamabe concernant la courbure scalaire.
\textit{J. Math. Pures Appl. (9)} \textbf{55} (1976), no. 3, 269–296.


\bibitem{Brendle}
S. Brendle,
Blow-up phenomena for the Yamabe equation. \textit{J. Amer. Math. Soc.} \textbf{21} (2008), no. 4, 951–979.

\bibitem{Brendle&Chen}
S. Brendle and S. S. Chen,
An existence theorem for the Yamabe problem on manifolds with boundary.
\textit{J. Eur. Math. Soc. (JEMS)} \textbf{16} (2014), no. 5, 991–1016.


\bibitem{Brendle&Marques}
S. Brendle and F. C. Marques,
Blow-up phenomena for the Yamabe equation. II. \textit{J. Differential Geom.} \textbf{81} (2009), no. 2, 225–250.


\bibitem{Chen&Ho&Sun}
X. Chen, P. T. Ho, and L. Sun,
Prescribed scalar curvature plus mean curvature flows in compact manifolds with boundary of negative conformal invariant.
\textit{Ann. Global Anal. Geom.} \textbf{53} (2018), no. 1, 121–150.

\bibitem{Chen&Ruan&Sun}
X. Chen, Y. Ruan, and L. Sun,
The Han-Li conjecture in constant scalar curvature and constant boundary mean curvature problem on compact manifolds.
\textit{Adv. Math.} \textbf{358} (2019), 106854, 56 pp.


\bibitem{Chen&Sun}
X. Chen and L. Sun,
Existence of conformal metrics with constant scalar curvature and constant boundary mean curvature on compact manifolds.
\textit{Commun. Contemp. Math.} \textbf{21} (2019), no. 3, 1850021, 51 pp.





\bibitem{Chen&Wu}
X. Chen and N. Wu,
Blow-up phenomena for the constant scalar curvature and constant boundary mean curvature equation. \textit{J. Differential Equations} \textbf{269} (2020), no. 11, 9432–9470.

\bibitem{Cherrier}
P. Cherrier,
Probl\`{e}mes de Neumann non lin\'{e}aires sur les vari\'{e}t\'{e}s riemanniennes.
\textit{J. Funct. Anal.} \textbf{57} (1984), no. 2, 154–206.


\bibitem{Disconzi&Khuri}
M. Disconzi and M. A. Khuri,  Compactness and non-compactness for the Yamabe problem on manifolds with boundary.
\textit{J. Reine Angew. Math.} \textbf{724} (2017), 145–201.


\bibitem{Druet}
O. Druet,
Compactness for Yamabe metrics in low dimensions.
\textit{Int. Math. Res. Not.} (2004), no. 23, 1143–1191.

\bibitem{Escobar1}
J. F. Escobar,
Conformal deformation of a Riemannian metric to a scalar flat metric with constant mean curvature on the boundary. \textit{Ann. of Math. (2)} \textbf{136} (1992), no. 1, 1–50.

\bibitem{Escobar2}
J. F. Escobar,  The Yamabe problem on manifolds with boundary. \textit{J. Differential Geom.} \textbf{35} (1992), no. 1, 21–84


\bibitem{Felli&Ould}
V. Felli and M. Ould Ahmedou,
Compactness results in conformal deformations of Riemannian metrics on manifolds with boundaries.
\textit{Math. Z.} \textbf{244} (2003), no. 1, 175–210.

\bibitem{Felli&Ould1}
V. Felli and M. Ould Ahmedou,
A geometric equation with critical nonlinearity on the boundary.
\textit{Pacific J. Math.} \textbf{218} (2005), no. 1, 75–99.

\bibitem{Ghimenti&Micheletti}
M. Ghimenti and A. M. Micheletti,
Compactness for conformal scalar-flat metrics on umbilic boundary manifolds.
\textit{Nonlinear Anal.} \textbf{200} (2020), 111992, 30 pp.


\bibitem{Han&Li1}
Z. C. Han and Y. Y. Li,
The existence of conformal metrics with constant scalar curvature and constant boundary mean curvature.
\textit{Comm. Anal. Geom.} \textbf{8} (2000), no. 4, 809–869.

\bibitem{Han&Li2}
Z. C. Han and Y. Y. Li,
The Yamabe problem on manifolds with boundary: existence and compactness results.
\textit{Duke Math. J.} \textbf{99} (1999), no. 3, 489–542.





\bibitem{Khuri&Marques&Schoen}
M. A. Khuri, F. C. Marques, and R. Schoen,
A compactness theorem for the Yamabe problem.
\textit{J. Differential Geom.} \textbf{81} (2009), no. 1, 143–196.

\bibitem{Kim&Musso&Wei}
S. Kim, M. Musso, and J. C. Wei,   Compactness of scalar-flat conformal metrics on low-dimensional manifolds with constant mean curvature on boundary. \textit{Ann. Inst. H. Poincar\'{e} C Anal. Non Lin\'{e}aire} \textbf{38} (2021), no. 6, 1763–1793.


\bibitem{Li&Zhu}
Y. Y. Li and M. Zhu,
Yamabe type equations on three-dimensional Riemannian manifolds.
\textit{Commun. Contemp. Math.} \textbf{1} (1999), no. 1, 1–50.


\bibitem{Marques}
F. C. Marques,
A priori estimates for the Yamabe problem in the non-locally conformally flat case.
\textit{J. Differential Geom.} \textbf{71} (2005), no. 2, 315–346.

\bibitem{Mayer&Ndiaye}
M. Mayer and C. B. Ndiaye,  Barycenter technique and the Riemann mapping problem of Cherrier-Escobar.
\textit{J. Differential Geom.} \textbf{107} (2017), no. 3, 519–560.


\bibitem{Pollack}
D. Pollack, Nonuniqueness and high energy solutions for a conformally invariant scalar equation.
\textit{Comm. Anal. Geom.} \textbf{1} (1993), no. 3-4, 347–414.

\bibitem{Schoen1}
R. Schoen,
A report on some recent progress on nonlinear problems in geometry. \textit{Surveys in differential geometry} (Cambridge, MA, 1990), 201–241, Lehigh Univ., Bethlehem, PA, 1991.

\bibitem{Schoen}
R. Schoen,
Conformal deformation of a Riemannian metric to constant scalar curvature.
\textit{J. Differential Geom.} \textbf{20} (1984), no. 2, 479–495.


\bibitem{Schoen2}
R. Schoen,
On the number of constant scalar curvature metrics in a conformal class. Differential geometry, 311–320,
\textit{Pitman Monogr. Surveys Pure Appl. Math.}, \textbf{52}, Longman Sci. Tech., Harlow, 1991.

\bibitem{Schoen3}
R. Schoen, Variational theory for the total scalar curvature functional for Riemannian metrics and related topics. Topics in calculus of variations (Montecatini Terme, 1987), 120–154,
\textit{Lecture Notes in Math.}, \textbf{1365}, Springer, Berlin, 1989.


\bibitem{Trudinger}
N. S. Trudinger,
Remarks concerning the conformal deformation of Riemannian structures on compact manifolds.
\textit{Ann. Scuola Norm. Sup. Pisa Cl. Sci. (3)} \textbf{22} (1968), 265–274.

\bibitem{code} The Mathematica script for this paper:
https://sites.google.com/view/paktungho/code

\end{thebibliography}

\end{document}